\documentclass{article}
\usepackage{amsfonts}
\usepackage{amsmath}

\newtheorem{theorem}{Theorem}

\newtheorem{definition}[theorem]{Definition}

\newtheorem{lemma}[theorem]{Lemma}

\newenvironment{proof}[1][Proof]{\noindent\textbf{#1.} }{\ \rule{0.5em}{0.5em}}
\input{tcilatex}
\begin{document}

\title{\textbf{Meromorphic Solutions of Homogeneous and Non-homogeneous
Higher Order Linear Difference Equations in Terms of (p,q)-Order}}
\author{Subhadip Khan$^{1}$, Chinmay Ghosh$^{2}$, Sanjib Kumar Datta$^{3}$%
\qquad  \\
$^{1}$Jotekamal High School\\
Village- Jotekamal, Block- Raghunathganj 2\\
Murshidabad, Pin-742133,\\
West Bengal, India\\
subhadip204@gmail.com\\
$^{2}$Department of Mathematics\\
Kazi Nazrul University\\
Nazrul Road, P.O.- Kalla C.H.\\
Asansol-713340, West Bengal, India \\
chinmayarp@gmail.com\\
$^{3}$Department of Mathematics\\
University of Kalyani\\
P.O.-Kalyani, Dist-Nadia, PIN-741235,\\
West Bengal, India\\
sanjibdatta05@gmail.com}
\date{}
\maketitle

\begin{abstract}
In this paper we investigate the growth of meromorphic solutions of
homogeneous and non-homogeneous linear difference equations with entire or
meromorphic coefficients. We further extend and improve few results on the
order of meromorphic solutions by using $(p,q)-$lower order and $(p,q)-$%
lower type followed by the investigation of Luo and Zheng $\left(
2016\right) ,$ Belaidi and Bellaama $\left( 2020\right) .$


\end{abstract}

\section{\protect\bigskip Introduction}

\bigskip The most common form of homogeneous and nonhomogeneous difference
equations are given by 
\begin{equation}
A_{k}(z)f(z+c_{k})+A_{k-1}(z)f(z+c_{k-1})+\cdots
+A_{1}(z)f(z+c_{1})+A_{0}(z)f(z)=0  \label{1h}
\end{equation}%
and%
\begin{equation}
A_{k}(z)f(z+c_{k})+A_{k-1}(z)f(z+c_{k-1})+\cdots
+A_{1}(z)f(z+c_{1})+A_{0}(z)f(z)=A_{k+1}(z),  \label{1nh}
\end{equation}%
where the coefficients $A_{0},A_{1},\ldots ,A_{k},A_{k+1}\neq 0$ $(k\geq 2)$
in (\ref{1h}) or (\ref{1nh}) are entire or meromorphic functions and $%
c_{k},c_{k-1},\ldots ,c_{1}$ are distinct nonzero complex numbers. Study of
the growth properties of meromorphic solutions of such equations from the
view point of Nevanlinna's theory is a matter of great interest for many
researchers in recent time.

In $2007$, Laine and Yang \cite{8} considered the homogeneous difference
equation\ of which more than one dominant coefficients exist but exactly one
has its type strictly greater than the others. Chiang and Feng \cite{3} in $%
2008$ established a theorem in which exactly one coefficient of (\ref{1h})
has maximal order. In $2013$, Liu and Mao \cite{9} investigated meromorphic
solutions of (\ref{1h}) or (\ref{1nh}) by using hyper order when one or more
coefficients of (\ref{1h}) or (\ref{1nh}) have infinite order. Also in $2017$%
, Zhou and Zheng and in $2019$, Bela\"{\i}di and Benkarouba \cite{1}
investigated the solutions of (\ref{1h}) or (\ref{1nh}) using iterated order
and iterated type which generalized and improved previous results. Very
recently, Luo and Zheng \cite{6a} have studied the growth of meromorphic
solutions of (\ref{1h}) when more than one coeffcient has maximal lower
order and the lower type strictly greater than the type of other
coeffcients, and obtained two very important Theorems and later in $2020,$
Bela\"{\i}di and Bellaama \cite{1ba} improved those two theorems of Luo and
Zheng \cite{6a} for nonhomogeneous case.

In $2010,$ Liu, Tu and Shi \cite{10} modified the definition of the $(p,q)-$%
order, which was first introduced by Juneja, Kapoor and Bajpai \cite{7}. In
this article we investigate meromorphic solutions of (\ref{1h}) and (\ref%
{1nh}) using the concept of $(p,q)-$lower order and $\left( p,q\right) -$%
lower type given in \cite{10}. We also extend and improve results of Luo and
Zheng \cite{6a}, Bela\"{\i}di and Bellaama \cite{1ba}. Here we consider the
case, when (\ref{1h}) and (\ref{1nh}) have meromorphic coefficients.

We assume that the reader is familiar with the fundamental results and the
standard notations of Nevanlinna's value distribution theory \cite{6aa}.

\section{Definitions and Lemmas}

In this section some basic definitions and lemmas are discussed for the
improvement of the main section.

Let $p\geq q\geq 1$ or $2\leq q=p+1$ be integers, then first recall the
following definitions :

\begin{definition}
\cite{8a} For a transcendental meromorphic function $f,$ the $(p,q)-$order
is defined by%
\begin{equation*}
\rho _{f}\left( p,q\right) =\limsup_{r\rightarrow \infty }\frac{\log
_{p}T(r,f)}{\log _{q}r}.
\end{equation*}

And if $f$ is a transcendental entire function, then%
\begin{equation*}
\rho _{f}\left( p,q\right) =\limsup_{r\rightarrow \infty }\frac{\log
_{p+1}M(r,f)}{\log _{q}r}.
\end{equation*}
\end{definition}

Note that $0\leq \rho _{f}\left( p,q\right) \leq \infty .$ Also for a
rational function $\rho _{f}\left( p,q\right) =0.$

\begin{definition}
\cite{8a} For a transcendental meromorphic function $f,$ the $(p,q)-$lower
order is defined by%
\begin{equation*}
\underline{\rho }_{f}\left( p,q\right) =\liminf\limits_{r\rightarrow \infty }%
\frac{\log _{p}T(r,f)}{\log _{q}r}.
\end{equation*}

And if $f$ is a transcendental entire function, then%
\begin{equation*}
\underline{\rho }_{f}\left( p,q\right) =\liminf\limits_{r\rightarrow \infty }%
\frac{\log _{p+1}M(r,f)}{\log _{q}r}.
\end{equation*}
\end{definition}

\begin{definition}
\cite{8a} A transcendental meromorphic function is said to have index pair $%
\left[ p,q\right] $ if $0\leq \rho _{f}\left( p,q\right) \leq \infty $ and $%
\rho _{f}\left( p-1,q-1\right) $ is not a non-zero finite number.
\end{definition}

\begin{definition}
\cite{8a} For a meromorphic function $f$ having non-zero finite $(p,q)-$%
order $\rho _{f}\left( p,q\right) ,$ the $(p,q)-$type is defined by%
\begin{equation*}
\tau _{f}\left( p,q\right) =\limsup_{r\rightarrow \infty }\frac{\log
_{p-1}T(r,f)}{\left( \log _{q-1}r\right) ^{\rho _{f}\left( p,q\right) }}.
\end{equation*}

And if $f$ is a transcendental entire function, then%
\begin{equation*}
\tau _{f}\left( p,q\right) =\limsup_{r\rightarrow \infty }\frac{\log
_{p}M(r,f)}{\left( \log _{q-1}r\right) ^{\rho _{f}\left( p,q\right) }}.
\end{equation*}
\end{definition}

\begin{definition}
\cite{8a} For a meromorphic function $f$ having non-zero finite $(p,q)-$%
lower order $\rho _{f}\left( p,q\right) ,$ the $(p,q)-$lower type is defined
by%
\begin{equation*}
\underline{\tau }_{f}\left( p,q\right) =\liminf\limits_{r\rightarrow \infty }%
\frac{\log _{p-1}T(r,f)}{\left( \log _{q-1}r\right) ^{\underline{\rho }%
_{f}\left( p,q\right) }}.
\end{equation*}

And if $f$ is a transcendental entire function, then%
\begin{equation*}
\underline{\tau }_{f}\left( p,q\right) =\liminf\limits_{r\rightarrow \infty }%
\frac{\log _{p}M(r,f)}{\left( \log _{q-1}r\right) ^{\underline{\rho }%
_{f}\left( p,q\right) }}.
\end{equation*}
\end{definition}

\begin{definition}
\cite{8a} For a meromorphic function $f$ , the $(p,q)-$exponent of
convergence of the sequence of poles is defined by%
\begin{equation*}
\lambda _{\frac{1}{f}}\left( p,q\right) =\limsup_{r\rightarrow \infty }\frac{%
\log _{p}N(r,f)}{\log _{q}r}.
\end{equation*}
\end{definition}

Where $N(r,f)$ denotes integrated counting function of poles of $f$ in the
domain $\left\{ z:\left\vert z\right\vert \leq r\right\} .$

We set $\exp _{1}r=e^{r}$ and $\exp _{k+1}r=\exp \left( \exp _{k}r\right) ,$ 
$k\in \mathbb{N}$ and consider $\exp _{0}r=\log _{0}r=r,$ $\exp _{-1}r=\log
_{1}r,$ $\exp _{1}r=\log _{-1}r,$ for all $r\in \mathbb{R}.$ Again for all
sufficiently large values of $r,$ $\log _{1}r=\log r$ and $\log _{k+1}r=\log
\left( \log _{k}r\right) ,$ $k\in \mathbb{N}.$

Next we recall some further definitions which are used throughout the main
section.

\begin{definition}
\cite{1} The linear measure of a set $E\subset \left( 0,+\infty \right) $ is
defined by 
\begin{equation*}
m\left( E\right) =\dint\limits_{0}^{\infty }\chi _{E}\left( t\right) dt
\end{equation*}%
and the logarithmic measure of a set $E\subset \left( 1,+\infty \right) $ is
defined by 
\begin{equation*}
lm\left( E\right) =\dint\limits_{1}^{\infty }\frac{\chi _{E}\left( t\right) 
}{t}dt,
\end{equation*}%
where $\chi _{E}\left( t\right) $ is the characteristic function of a set $%
E. $
\end{definition}

Now we recall some lemmas. Some of which are proved in terms of generalized $%
(p,q)-$order.

\begin{lemma}
\label{L1}\cite{5aa} Let $f$ be a meromorphic function of generalized $%
\left( p,q\right) $-order $\rho _{f}\left( p,q\right) =\rho $ and suppose $%
\xi _{1}$ and $\xi _{2}$ be two arbitrary distinct complex numbers. Then for
given $\varepsilon >0,$ there exists a subset $S_{1}\subset \left( 1,+\infty
\right) $ of finite logarithmic measure such that for all $\left\vert
z\right\vert =r\notin S_{1}\cup \left[ 0,1\right] ,$ we have 
\begin{equation*}
i)\text{ }\exp \left\{ -r^{\rho -1+\varepsilon }\right\} \leq \left\vert 
\frac{f\left( z+\xi _{1}\right) }{f\left( z+\xi _{2}\right) }\right\vert
\leq \exp \left\{ r^{\rho -1+\varepsilon }\right\} ,
\end{equation*}

for $p=q=1,$ and%
\begin{equation*}
ii)\text{ }\exp _{p}\left\{ -\log _{q-1}r^{\rho _{f}\left( p,q\right)
+\varepsilon }\right\} \leq \left\vert \frac{f\left( z+\xi _{1}\right) }{%
f\left( z+\xi _{2}\right) }\right\vert \leq \exp _{p}\left\{ \log
_{q-1}r^{\rho _{f}\left( p,q\right) +\varepsilon }\right\} ,
\end{equation*}

for $p\geq q\geq 2.$
\end{lemma}

\begin{lemma}
\label{L2}\cite{4} Let $f$ be a nonconstant meromorphic function and $\xi $
be nonzero complex constants. Then for $r\rightarrow +\infty $ we have%
\begin{equation*}
\left( 1+o(1)\right) T\left( r-\left\vert \xi \right\vert ,\text{ }f\right)
\leq T\left( r,\text{ }f\left( z+\xi \right) \right) \leq \left(
1+o(1)\right) T\left( r+\left\vert \xi \right\vert ,\text{ }f\right) .
\end{equation*}
\end{lemma}

\begin{lemma}
\label{L3}\cite{5aa} Let $f$ be a meromorphic function of generalized $%
\left( p,q\right) $-order $\rho _{f}\left( p,q\right) =\rho <+\infty $ and
suppose $\xi _{1}$ and $\xi _{2}$ be two arbitrary distinct complex numbers.
Then for each $\varepsilon >0,$ we have%
\begin{equation*}
i)\text{ }m\left( r,\frac{f\left( z+\xi _{1}\right) }{f\left( z+\xi
_{2}\right) }\right) =O\left( r^{\rho -1+\varepsilon }\right) ,
\end{equation*}

for $p=q=1,$ and%
\begin{equation*}
ii)\text{ }m\left( r,\frac{f\left( z+\xi _{1}\right) }{f\left( z+\xi
_{2}\right) }\right) =O\left( \exp _{p-1}\left[ \left\{ \log _{q-1}\left(
r\right) \right\} ^{\rho _{f}\left( p,q\right) +\varepsilon }\right] \right)
,
\end{equation*}

for $p\geq q\geq 2.$
\end{lemma}

\begin{lemma}
\label{L4} Let $f(z)$ be a meromorphic function with $\rho _{f}\left(
p,q\right) <\infty .$ Then for any given $\varepsilon >0,$ there exists a
subset $S_{2}\subset \left( 1,+\infty \right) $ with finite linear measure
and finite logarithmic measure such that for all $\left\vert z\right\vert
=r\notin S_{2}\cup \left[ 0,1\right] $ and for sufficiently large $r,$ we
have%
\begin{equation*}
\exp _{p}\left[ -\left\{ \log _{q-1}\left( r\right) \right\} ^{\rho
_{f}\left( p,q\right) +\varepsilon }\right] \leq \left\vert f(z)\right\vert
\leq \exp _{p}\left[ \left\{ \log _{q-1}\left( r\right) \right\} ^{\rho
_{f}\left( p,q\right) +\varepsilon }\right] .
\end{equation*}
\end{lemma}

\begin{proof}
The lemma hold trivially if $f(z)$ has finitely many poles.

For the case when $f(z)$ has infinitely many poles, the lemma follows from 
\cite{1aa} and \cite{1b}. So we are skipping the proof.
\end{proof}

\begin{lemma}
\label{L5} Let $f$ be an entire function with $\underline{\rho }_{f}\left(
p,q\right) <\infty .$ Then for any given $\varepsilon >0,$ there exists a
subset $S_{3}\subset \left( 1,+\infty \right) $ with infinite logarithmic
measure such that for all $r\in S_{3},$ we have%
\begin{equation*}
\underline{\rho }_{f}\left( p,q\right) =\lim_{\substack{ r\rightarrow
+\infty  \\ r\in S_{3}}}\frac{\log _{p+1}M(r,f)}{\log _{q}r}
\end{equation*}

and%
\begin{equation*}
M\left( r,f\right) <\exp _{p}\left[ \left\{ \log _{q-1}\left( r\right)
\right\} ^{\underline{\rho }_{f}\left( p,q\right) +\varepsilon }\right] .
\end{equation*}
\end{lemma}

\begin{proof}
For the proof we use similar logic described in \cite{5a}.

By the definition of ($p,q)-$order of an entire function $f(z)$, there
exists a sequence $\left\{ r_{n}\right\} _{n=1}^{\infty }$ tending to
infinity in such a way that $\left( 1+\frac{1}{n}\right) r_{n}<r_{n+1}.$

Therefore we have%
\begin{equation*}
\underline{\rho }_{f}\left( p,q\right) =\lim_{r_{n}\rightarrow +\infty }%
\frac{\log _{p+1}M(r_{n},f)}{\log _{q}r_{n}}.
\end{equation*}

Then for any chosen $\varepsilon >0,$ there exist $n_{1}$ such that for $%
n\geq n_{1}$ and for $r\in \left[ r_{n},\left( 1+\frac{1}{n}\right) r_{n}%
\right] ,$ we have%
\begin{equation*}
\frac{\log _{p+1}M(r_{n},f)}{\log _{q}\left( 1+\frac{1}{n}\right) r_{n}}\leq 
\frac{\log _{p+1}M(r,f)}{\log _{q}r}\leq \frac{\log _{p+1}M(\left( 1+\frac{1%
}{n}\right) r_{n},f)}{\log _{q}r_{n}}.
\end{equation*}

Let $S_{3}=\dbigcup\limits_{n=n_{1}}^{\infty }\left[ r_{n},\left( 1+\frac{1}{%
n}\right) r_{n}\right] ,$ then for any $r\in S_{3}$ we have%
\begin{equation*}
\underline{\rho }_{f}\left( p,q\right) =\lim_{r_{n}\rightarrow +\infty }%
\frac{\log _{p+1}M(r_{n},f)}{\log _{q}r_{n}}=\lim_{\substack{ r\rightarrow
+\infty  \\ r\in S_{3}}}\frac{\log _{p+1}M(r,f)}{\log _{q}r}
\end{equation*}

and logarithmic measure of $S_{3}$ is $\dsum\limits_{n=n_{1}}^{\infty
}\dint\limits_{r_{n}}^{\left( 1+\frac{1}{n}\right) r_{n}}\frac{dz}{z}=$ $%
\dsum\limits_{n=n_{1}}^{\infty }\log \left( 1+\frac{1}{n}\right) =\infty .$

Second part of the lemma follows from the definition stated in the lemma.
\end{proof}

\begin{lemma}
\label{L6} Let $f$ be an entire function with $0<\underline{\rho }_{f}\left(
p,q\right) <\infty .$ Then for any given $\varepsilon >0,$ there exists a
subset $S_{4}\subset \left( 1,+\infty \right) $ with infinite logarithmic
measure such that for all $r\in S_{4},$ we have%
\begin{equation*}
\underline{\tau }_{f}\left( p,q\right) =\lim_{\substack{ r\rightarrow
+\infty  \\ r\in S_{4}}}\frac{\log _{p}M(r,f)}{\left( \log _{q-1}r\right) ^{%
\underline{\rho }_{f}\left( p,q\right) }}
\end{equation*}

and%
\begin{equation*}
M\left( r,f\right) <\exp _{p}\left[ \left( \underline{\tau }_{f}\left(
p,q\right) +\varepsilon \right) \left\{ \log _{q-1}\left( r\right) \right\}
^{\underline{\rho }_{f}\left( p,q\right) }\right] .
\end{equation*}
\end{lemma}

\begin{proof}
For the proof we use similar logic described in \cite{5a}.

By the definition of $(p,q)-$order, there exists a sequence $\left\{
r_{n}\right\} _{n=1}^{\infty }$ tending to infinity in such a way that $%
\left( 1+\frac{1}{n}\right) r_{n}<r_{n+1}.$

Therefore we have%
\begin{equation*}
\underline{\tau }_{f}\left( p,q\right) =\lim_{r_{n}\rightarrow +\infty }%
\frac{\log _{p}M(r_{n},f)}{\left( \log _{q-1}r_{n}\right) ^{\underline{\rho }%
_{f}\left( p,q\right) }}.
\end{equation*}

Then for any chosen $\varepsilon >0,$ there exists $n_{1}$ such that for $%
n\geq n_{1}$ and for $r\in \left[ r_{n},\left( 1+\frac{1}{n}\right) r_{n}%
\right] ,$ we have%
\begin{equation*}
\frac{\log _{p}M(r_{n},f)}{\left( \log _{q-1}\left( 1+\frac{1}{n}\right)
r_{n}\right) ^{\underline{\rho }_{f}\left( p,q\right) }}\leq \frac{\log
_{p}M(r,f)}{\left( \log _{q-1}r\right) ^{\underline{\rho }_{f}\left(
p,q\right) }}\leq \frac{\log _{p}M(\left( 1+\frac{1}{n}\right) r_{n},f)}{%
\left( \log _{q-1}r_{n}\right) ^{\underline{\rho }_{f}\left( p,q\right) }}.
\end{equation*}

Let $S_{4}=\dbigcup\limits_{n=n_{1}}^{\infty }\left[ r_{n},\left( 1+\frac{1}{%
n}\right) r_{n}\right] ,$ then for any $r\in S_{4}$ we have%
\begin{equation*}
\underline{\rho }_{f}\left( p,q\right) =\lim_{r_{n}\rightarrow +\infty }%
\frac{\log _{p}M(r_{n},f)}{\left( \log _{q-1}r_{n}\right) ^{\underline{\rho }%
_{f}\left( p,q\right) }}=\lim_{\substack{ r\rightarrow +\infty  \\ r\in
S_{4} }}\frac{\log _{p}M(r,f)}{\left( \log _{q-1}r\right) ^{\underline{\rho }%
_{f}\left( p,q\right) }}
\end{equation*}

and logarithmic measure of $S_{4}$ is $\dsum\limits_{n=n_{1}}^{\infty
}\dint\limits_{r_{n}}^{\left( 1+\frac{1}{n}\right) r_{n}}\frac{dz}{z}=$ $%
\dsum\limits_{n=n_{1}}^{\infty }\log \left( 1+\frac{1}{n}\right) =\infty .$

Second part of the lemma follows from the definition stated in the lemma.
\end{proof}

\begin{lemma}
\label{L7} Let $f$ be a meromorphic function with $\underline{\rho }%
_{f}\left( p,q\right) <\infty .$ Then for any given $\varepsilon >0,$ there
exists a subset $S_{5}\subset \left( 1,+\infty \right) $ with infinite
logarithmic measure such that for all $r\in S_{5},$ we have%
\begin{equation*}
\underline{\rho }_{f}\left( p,q\right) =\lim_{\substack{ r\rightarrow
+\infty  \\ r\in S_{5}}}\frac{\log _{p}T(r,f)}{\log _{q}r}
\end{equation*}

and%
\begin{equation*}
T\left( r,f\right) <\exp _{p-1}\left[ \left\{ \log _{q-1}\left( r\right)
\right\} ^{\underline{\rho }_{f}\left( p,q\right) +\varepsilon }\right] .
\end{equation*}
\end{lemma}

\begin{proof}
Proof of the lemma follows in a same way as we have done in Lemma $\ref{5}$.
\end{proof}

\begin{lemma}
\label{L8} Let $f$ be a meromorphic function with $0<\underline{\rho }%
_{f}\left( p,q\right) <\infty .$ Then for any given $\varepsilon >0,$ there
exists a subset $S_{6}\subset \left( 1,+\infty \right) $ with infinite
logarithmic measure such that for all $r\in S_{6},$ we have%
\begin{equation*}
\underline{\tau }_{f}\left( p,q\right) =\lim_{\substack{ r\rightarrow
+\infty  \\ r\in S_{6}}}\frac{\log _{p-1}T(r,f)}{\left( \log _{q-1}r\right)
^{\underline{\rho }_{f}\left( p,q\right) }}
\end{equation*}

and%
\begin{equation*}
T\left( r,f\right) <\exp _{p-1}\left[ \left( \underline{\tau }_{f}\left(
p,q\right) +\varepsilon \right) \left\{ \log _{q-1}\left( r\right) \right\}
^{\underline{\rho }_{f}\left( p,q\right) }\right] .
\end{equation*}
\end{lemma}

\begin{proof}
Proof of the lemma follows in a same way as we have done in Lemma $\ref{6}$.
\end{proof}

\section{Main Results}

Now we prove our main results of this paper.

\begin{theorem}
\label{T1} Let $A_{0}(z),A_{1}(z),....,A_{k}(z),$ $A_{k+1}(z)$ be entire
functions of generalize $\left( p,q\right) $-order. If there exist integers $%
l,m$ $\left( 0\leq l,m\leq k+1\right) $ and

the following conditions hold simultaneously

(1) $0<\max \left\{ \underline{\rho }_{A_{m}}(p,q),\rho _{A_{j}}(p,q),j\neq
m,l\right\} =\rho \left( p,q\right) \leq \underline{\rho }%
_{A_{l}}(p,q)<\infty ;$

(2) \underline{$\tau $}$_{A_{l}}(p,q)>$\underline{$\tau $}$_{A_{m}}(p,q),$
when $\underline{\rho }_{A_{l}}(p,q)=\underline{\rho }_{A_{m}}(p,q);$

(3) $\max \left\{ \tau _{A_{j}}(p,q):\rho _{A_{j}}(p,q)=\underline{\rho }%
_{A_{l}}(p,q),\text{ }j\neq m,l\right\} =\tau _{1}<$ \underline{$\tau $}$%
_{A_{l}}(p,q),$ when $\underline{\rho }_{A_{l}}(p,q)=\max \left\{ \rho
_{A_{j}}(p,q),j\neq m,l\right\} .$

Then every meromorphic solution $f\left( \neq 0\right) $ of equation $\left( %
\ref{1nh}\right) $ satisfies $\rho _{f}(p,q)\geq \underline{\rho }%
_{A_{l}}(p,q),$ where $p\geq q\geq 2.$
\end{theorem}

\begin{proof}
First let $f$ $\left( \neq 0\right) $ be a meromorphic solution of $\left( %
\ref{1nh}\right) $ and divide $\left( \ref{1nh}\right) $ by $f(z+c_{l})$ we
get%
\begin{equation*}
-A_{l}(z)=A_{k}(z)\frac{f(z+c_{k})}{f(z+c_{l})}+...+A_{l-1}(z)\frac{%
f(z+c_{l-1})}{f(z+c_{l})}+...+A_{1}(z)\frac{f(z+c_{1})}{f(z+c_{l})}+A_{0}(z)%
\frac{f(z)}{f(z+c_{l})}-\frac{A_{k+1}(z)}{f(z+c_{l})}.
\end{equation*}%
The above expression can be written as%
\begin{equation}
\left\vert A_{l}(z)\right\vert \leq \dsum\limits_{j=1,j\neq
m,l}^{k}\left\vert A_{j}(z)\right\vert \left\vert \frac{f(z+c_{j})}{%
f(z+c_{l})}\right\vert +\left\vert A_{m}(z)\right\vert \left\vert \frac{%
f(z+c_{m})}{f(z+c_{l})}\right\vert +\left\vert A_{0}(z)\right\vert
\left\vert \frac{f(z)}{f(z+c_{l})}\right\vert +\left\vert \frac{A_{k+1}(z)}{%
f(z+c_{l})}\right\vert .  \label{1}
\end{equation}

By Lemma $\ref{L1}$ $\left( ii\right) $, for any given $\varepsilon >0$,
there exists a subset $S_{1}\subset \left( 1,+\infty \right) $ of finite
logarithmic measure such that for all $\left\vert z\right\vert =r\notin
S_{1}\cup \left[ 0,1\right] ,$ we have%
\begin{equation}
\left\vert \frac{f(z+c_{j})}{f(z+c_{l})}\right\vert \leq \exp _{p}\left[
\left\{ \log _{q-1}\left( r\right) \right\} ^{\rho _{f}\left( p,q\right)
+\varepsilon }\right] ,\left( j\neq l\right)  \label{2}
\end{equation}%
and%
\begin{equation}
\left\vert \frac{f(z)}{f(z+c_{l})}\right\vert \leq \exp _{p}\left[ \left\{
\log _{q-1}\left( r\right) \right\} ^{\rho _{f}\left( p,q\right)
+\varepsilon }\right] .  \label{3}
\end{equation}

Again by Lemma $\ref{L2},$ we have $\rho \left( f(z+c_{l})\right) =\rho
\left( \frac{1}{f(z+c_{l})}\right) =\rho (f).$

Then by Lemma $\ref{L4},$ for the above $\varepsilon ,$ there exists a
subset $S_{2}\subset \left( 1,+\infty \right) $ of finite logarithmic
measure such that for all $\left\vert z\right\vert =r\notin S_{2}\cup \left[
0,1\right] ,$ we have%
\begin{equation}
\left\vert \frac{1}{f(z+c_{l})}\right\vert \leq \exp _{p}\left[ \left\{ \log
_{q-1}\left( r\right) \right\} ^{\rho _{f}\left( p,q\right) +\varepsilon }%
\right] .  \label{4}
\end{equation}

Now the following four cases arrive, namely;

\textbf{Case 1.} Consider the case when $\rho \left( p,q\right) <\underline{%
\rho }_{A_{l}}(p,q).$

Now for sufficiently large $r$ and for above $\varepsilon $ by definition we
have%
\begin{equation}
\left\vert A_{j}(z)\right\vert \leq \exp _{p}\left[ \left\{ \log
_{q-1}\left( r\right) \right\} ^{\rho _{A_{j}}(p,q)+\varepsilon }\right]
\leq \exp _{p}\left[ \left\{ \log _{q-1}\left( r\right) \right\} ^{\rho
(p,q)+\varepsilon }\right] ,j\neq l,m  \label{5}
\end{equation}

and%
\begin{equation}
\left\vert A_{l}(z)\right\vert \geq \exp _{p}\left[ \left\{ \log
_{q-1}\left( r\right) \right\} ^{\underline{\rho }_{A_{l}}(p,q)-\varepsilon }%
\right] .  \label{6}
\end{equation}

Again from Lemma $\ref{L5}$ by using the definition of $\underline{\rho }%
_{A_{m}},$for any $\varepsilon >0,$ there exists a subset $S_{3}\subset
\left( 1,+\infty \right) $ of infinite logarithmic measure such that for all 
$\left\vert z\right\vert =r\in S_{3},$ we have%
\begin{equation}
\left\vert A_{m}(z)\right\vert \leq \exp _{p}\left[ \left\{ \log
_{q-1}\left( r\right) \right\} ^{\underline{\rho }_{A_{m}}(p,q)+\varepsilon }%
\right] .  \label{7}
\end{equation}

Hence by substituting $\left( \ref{2}\right) -\left( \ref{7}\right) $ in $%
\left( \ref{1}\right) $, for all $\left\vert z\right\vert =r\in S_{3}$ $%
\backslash $ $\left( \left[ 0,1\right] \cup S_{1}\cup S_{2}\right) ,$ we have%
\begin{eqnarray*}
\exp _{p}\left[ \left\{ \log _{q-1}\left( r\right) \right\} ^{\underline{%
\rho }_{A_{l}}(p,q)-\varepsilon }\right] &\leq &(k-2)\exp _{p}\left[ \left\{
\log _{q-1}\left( r\right) \right\} ^{\rho (p,q)+\varepsilon }\right] .\exp
_{p}\left[ \left\{ \log _{q-1}\left( r\right) \right\} ^{\rho _{f}\left(
p,q\right) +\varepsilon }\right] \\
&&+\exp _{p}\left[ \left\{ \log _{q-1}\left( r\right) \right\} ^{\underline{%
\rho }_{A_{m}}(p,q)+\varepsilon }\right] .\exp _{p}\left[ \left\{ \log
_{q-1}\left( r\right) \right\} ^{\rho _{f}\left( p,q\right) +\varepsilon }%
\right] \\
&&+\exp _{p}\left[ \left\{ \log _{q-1}\left( r\right) \right\} ^{\rho
(p,q)+\varepsilon }\right] .\exp _{p}\left[ \left\{ \log _{q-1}\left(
r\right) \right\} ^{\rho _{f}\left( p,q\right) +\varepsilon }\right] .
\end{eqnarray*}

Choose sufficiently small $\varepsilon $ so that $0<3\varepsilon <\underline{%
\rho }_{A_{l}}(p,q)-\rho (p,q)$ and for all $\left\vert z\right\vert =r\in
S_{3}$ $\backslash $ $\left( \left[ 0,1\right] \cup S_{1}\cup S_{2}\right)
,r\rightarrow \infty $ we have%
\begin{equation*}
\exp _{p}\left[ \left\{ \log _{q-1}\left( r\right) \right\} ^{\underline{%
\rho }_{A_{l}}(p,q)-2\varepsilon }\right] \leq \exp _{p}\left[ \left\{ \log
_{q-1}\left( r\right) \right\} ^{\rho _{f}\left( p,q\right) +\varepsilon }%
\right] .
\end{equation*}

Therefore%
\begin{equation*}
\underline{\rho }_{A_{l}}(p,q)\leq \rho _{f}\left( p,q\right) +3\varepsilon .
\end{equation*}

Since $\varepsilon >0$ is arbitrary, we have $\rho _{f}\left( p,q\right)
\geq \underline{\rho }_{A_{l}}(p,q).$

If $A_{k+1}=0,$ by substituting $\left( \ref{2}\right) -\left( \ref{7}%
\right) $ in $\left( \ref{1}\right) $, for all $\left\vert z\right\vert
=r\in S_{3}$ $\backslash $ $\left( \left[ 0,1\right] \cup S_{1}\cup
S_{2}\right) ,$ we have%
\begin{eqnarray*}
\exp _{p}\left[ \left\{ \log _{q-1}\left( r\right) \right\} ^{\underline{%
\rho }_{A_{l}}(p,q)-\varepsilon }\right] &\leq &(k-1)\exp _{p}\left[ \left\{
\log _{q-1}\left( r\right) \right\} ^{\rho (p,q)+\varepsilon }\right] .\exp
_{p}\left[ \left\{ \log _{q-1}\left( r\right) \right\} ^{\rho _{f}\left(
p,q\right) +\varepsilon }\right] \\
&&+\exp _{p}\left[ \left\{ \log _{q-1}\left( r\right) \right\} ^{\underline{%
\rho }_{A_{m}}(p,q)+\varepsilon }\right] .\exp _{p}\left[ \left\{ \log
_{q-1}\left( r\right) \right\} ^{\rho _{f}\left( p,q\right) +\varepsilon }%
\right] \\
&&+\exp _{p}\left[ \left\{ \log _{q-1}\left( r\right) \right\} ^{\rho
(p,q)+\varepsilon }\right] .\exp _{p}\left[ \left\{ \log _{q-1}\left(
r\right) \right\} ^{\rho _{f}\left( p,q\right) +\varepsilon }\right] .
\end{eqnarray*}

Choose sufficiently small $\varepsilon $ so that $0<3\varepsilon <\underline{%
\rho }_{A_{l}}(p,q)-\rho (p,q)$ and for all $\left\vert z\right\vert =r\in
S_{3}$ $\backslash $ $\left( \left[ 0,1\right] \cup S_{1}\cup S_{2}\right)
,r\rightarrow \infty $ we have%
\begin{equation*}
\exp _{p}\left[ \left\{ \log _{q-1}\left( r\right) \right\} ^{\underline{%
\rho }_{A_{l}}(p,q)-2\varepsilon }\right] \leq \exp _{p}\left[ \left\{ \log
_{q-1}\left( r\right) \right\} ^{\rho _{f}\left( p,q\right) +\varepsilon }%
\right] .
\end{equation*}

Therefore%
\begin{equation*}
\underline{\rho }_{A_{l}}(p,q)\leq \rho _{f}\left( p,q\right) +3\varepsilon .
\end{equation*}

Since $\varepsilon >0$ is arbitrary, we have $\rho _{f}\left( p,q\right)
\geq \underline{\rho }_{A_{l}}(p,q).$

\textbf{Case 2.} \ Consider $\max \left\{ \rho _{A_{j}}(p,q),j\neq
m,l\right\} =\alpha <\underline{\rho }_{A_{m}}(p,q)=\underline{\rho }%
_{A_{l}}(p,q),$ \underline{$\tau $}$_{A_{l}}(p,q)>$\underline{$\tau $}$%
_{A_{m}}(p,q).$

Now for sufficiently large $r$ and for above $\varepsilon $ by definition we
have%
\begin{equation}
\left\vert A_{j}(z)\right\vert \leq \exp _{p}\left[ \left\{ \log
_{q-1}\left( r\right) \right\} ^{\rho _{A_{j}}(p,q)+\varepsilon }\right]
\leq \exp _{p}\left[ \left\{ \log _{q-1}\left( r\right) \right\} ^{\alpha
+\varepsilon }\right] ,j\neq l,m  \label{8}
\end{equation}%
and%
\begin{equation}
\left\vert A_{l}(z)\right\vert \geq \exp _{p}\left[ \left( \underline{\tau }%
_{A_{l}}(p,q)-\varepsilon \right) \left\{ \log _{q-1}\left( r\right)
\right\} ^{\underline{\rho }_{A_{l}}(p,q)}\right] .  \label{9}
\end{equation}

Again by the Lemma $\ref{L6}$ using the definition of $\underline{\tau }%
_{A_{m}},$for any $\varepsilon >0,$ there exists a subset $S_{4}\subset
\left( 1,+\infty \right) $ of infinite logarithmic measure such that for all 
$\left\vert z\right\vert =r\in S_{4},$ we have%
\begin{eqnarray}
\left\vert A_{m}(z)\right\vert &\leq &\exp _{p}\left[ \left( \underline{\tau 
}_{A_{m}}(p,q)+\varepsilon \right) \left\{ \log _{q-1}\left( r\right)
\right\} ^{\underline{\rho }_{A_{m}}(p,q)}\right]  \label{10} \\
&=&\exp _{p}\left[ \left( \underline{\tau }_{A_{m}}(p,q)+\varepsilon \right)
\left\{ \log _{q-1}\left( r\right) \right\} ^{\underline{\rho }_{A_{l}}(p,q)}%
\right] .
\end{eqnarray}

Hence by substituting $\left( \ref{2}\right) -\left( \ref{4}\right) ,\left( %
\ref{8}\right) -\left( \ref{10}\right) $ in $\left( \ref{1}\right) $, for
all $\left\vert z\right\vert =r\in S_{4}$ $\backslash $ $\left( \left[ 0,1%
\right] \cup S_{1}\cup S_{2}\right) ,$ we have%
\begin{eqnarray*}
&&\exp _{p}\left[ \left( \underline{\tau }_{A_{l}}(p,q)-\varepsilon \right)
\left\{ \log _{q-1}\left( r\right) \right\} ^{\underline{\rho }_{A_{l}}(p,q)}%
\right] \\
&\leq &(k-1)\exp _{p}\left[ \left\{ \log _{q-1}\left( r\right) \right\}
^{\alpha +\varepsilon }\right] .\exp _{p}\left[ \left\{ \log _{q-1}\left(
r\right) \right\} ^{\rho _{f}\left( p,q\right) +\varepsilon }\right] \\
&&+\exp _{p}\left[ \left( \underline{\tau }_{A_{m}}(p,q)+\varepsilon \right)
\left\{ \log _{q-1}\left( r\right) \right\} ^{\underline{\rho }_{A_{l}}(p,q)}%
\right] .\exp _{p}\left[ \left\{ \log _{q-1}\left( r\right) \right\} ^{\rho
_{f}\left( p,q\right) +\varepsilon }\right] \\
&&+\exp _{p}\left[ \left\{ \log _{q-1}\left( r\right) \right\} ^{\alpha
+\varepsilon }\right] .\exp _{p}\left[ \left\{ \log _{q-1}\left( r\right)
\right\} ^{\rho _{f}\left( p,q\right) +\varepsilon }\right] .
\end{eqnarray*}

Choose sufficiently small $\varepsilon $ so that $0<2\varepsilon <\min
\left\{ \underline{\rho }_{A_{l}}(p,q)-\alpha ,\underline{\tau }%
_{A_{l}}(p,q)-\underline{\tau }_{A_{m}}(p,q)\right\} $ and

for all $\left\vert z\right\vert =r\in S_{4}$ $\backslash $ $\left( \left[
0,1\right] \cup S_{1}\cup S_{2}\right) ,r\rightarrow \infty $ we have%
\begin{equation*}
\exp _{p}\left[ \left( \underline{\tau }_{A_{l}}(p,q)-\underline{\tau }%
_{A_{m}}(p,q)-2\varepsilon \right) \left\{ \log _{q-1}\left( r\right)
\right\} ^{\underline{\rho }_{A_{l}}(p,q)-\varepsilon }\right] \leq \exp _{p}%
\left[ \left\{ \log _{q-1}\left( r\right) \right\} ^{\rho _{f}\left(
p,q\right) +\varepsilon }\right] ,
\end{equation*}

which implies%
\begin{equation*}
\underline{\rho }_{A_{l}}(p,q)\leq \rho _{f}\left( p,q\right) +2\varepsilon .
\end{equation*}

Since $\varepsilon >0$ is arbitrary, therefore $\rho _{f}\left( p,q\right)
\geq \underline{\rho }_{A_{l}}(p,q).$

\textbf{Case 3.} Suppose $\underline{\rho }_{A_{m}}(p,q)<\max \left\{ \rho
_{A_{j}}(p,q),\text{ }j\neq m,l\right\} =\underline{\rho }_{A_{l}}(p,q)$ and

$\max \left\{ \tau _{A_{j}}(p,q):\rho _{A_{j}}(p,q)=\underline{\rho }%
_{A_{l}}(p,q),\text{ }j\neq m,l\right\} =\tau _{1}<$ \underline{$\tau $}$%
_{A_{l}}(p,q).$

Now for sufficiently large $r$ and for above $\varepsilon $ by definition we
have%
\begin{eqnarray}
\left\vert A_{j}(z)\right\vert &\leq &\exp _{p}\left[ \left\{ \log
_{q-1}\left( r\right) \right\} ^{\rho _{A_{j}}(p,q)+\varepsilon }\right] 
\notag \\
&\leq &\exp _{p}\left[ \left\{ \log _{q-1}\left( r\right) \right\} ^{%
\underline{\rho }_{A_{l}}(p,q)-\varepsilon }\right] ,\text{ if }\rho
_{A_{j}}(p,q)<\underline{\rho }_{A_{l}}(p,q),j\neq l,m  \label{11}
\end{eqnarray}%
and%
\begin{equation}
\left\vert A_{j}(z)\right\vert \leq \exp _{p}\left[ \left( \tau
_{1}+\varepsilon \right) \left\{ \log _{q-1}\left( r\right) \right\} ^{%
\underline{\rho }_{A_{l}}(p,q)}\right] ,\text{if }\rho _{A_{j}}(p,q)=%
\underline{\rho }_{A_{l}}(p,q),j\neq l,m.  \label{12}
\end{equation}

Hence by substituting $\left( \ref{2}\right) -\left( \ref{4}\right) ,\left( %
\ref{7}\right) ,\left( \ref{9}\right) ,\left( \ref{11}\right) ,\left( \ref%
{12}\right) $ in $\left( \ref{1}\right) $, for all $\left\vert z\right\vert
=r\in S_{3}$ $\backslash $ $\left( \left[ 0,1\right] \cup S_{1}\cup
S_{2}\right) ,$ we have%
\begin{eqnarray*}
&&\exp _{p}\left[ \left( \underline{\tau }_{A_{l}}(p,q)-\varepsilon \right)
\left\{ \log _{q-1}\left( r\right) \right\} ^{\underline{\rho }_{A_{l}}(p,q)}%
\right] \\
&\leq &O\left( \exp _{p}\left[ \left( \tau _{1}+\varepsilon \right) \left\{
\log _{q-1}\left( r\right) \right\} ^{\underline{\rho }_{A_{l}}(p,q)}\right]
.\exp _{p}\left[ \left\{ \log _{q-1}\left( r\right) \right\} ^{\rho
_{f}\left( p,q\right) +\varepsilon }\right] \right) \\
&&+O\left( \exp _{p}\left[ \left\{ \log _{q-1}\left( r\right) \right\} ^{%
\underline{\rho }_{A_{l}}(p,q)-\varepsilon }\right] .\exp _{p}\left[ \left\{
\log _{q-1}\left( r\right) \right\} ^{\rho _{f}\left( p,q\right)
+\varepsilon }\right] \right) \\
&&+\exp _{p}\left[ \left\{ \log _{q-1}\left( r\right) \right\} ^{\underline{%
\rho }_{Am}(p,q)+\varepsilon }\right] .\exp _{p}\left[ \left\{ \log
_{q-1}\left( r\right) \right\} ^{\rho _{f}\left( p,q\right) +\varepsilon }%
\right] \\
&&+\exp _{p}\left[ \left( \tau _{1}+\varepsilon \right) \left\{ \log
_{q-1}\left( r\right) \right\} ^{\underline{\rho }_{A_{l}}(p,q)}\right]
.\exp _{p}\left[ \left\{ \log _{q-1}\left( r\right) \right\} ^{\rho
_{f}\left( p,q\right) +\varepsilon }\right] .
\end{eqnarray*}

Choose sufficiently small $\varepsilon $ so that $0<2\varepsilon <\min
\left\{ \underline{\rho }_{A_{l}}(p,q)-\underline{\rho }_{A_{m}}(p,q),%
\underline{\tau }_{A_{l}}(p,q)-\tau _{1}\right\} $and

for all $\left\vert z\right\vert =r\in S_{3}$ $\backslash $ $\left( \left[
0,1\right] \cup S_{1}\cup S_{2}\right) ,r\rightarrow \infty $ we have%
\begin{equation*}
\exp _{p}\left[ \left( \underline{\tau }_{A_{l}}(p,q)-\tau _{1}-2\varepsilon
\right) \left\{ \log _{q-1}\left( r\right) \right\} ^{\underline{\rho }%
_{A_{l}}(p,q)-\varepsilon }\right] \leq \exp _{p}\left[ \left\{ \log
_{q-1}\left( r\right) \right\} ^{\rho _{f}\left( p,q\right) +\varepsilon }%
\right] ,
\end{equation*}

which implies%
\begin{equation*}
\underline{\rho }_{A_{l}}(p,q)\leq \rho _{f}\left( p,q\right) +2\varepsilon .
\end{equation*}

Since $\varepsilon >0$ is arbitrary, therefore $\rho _{f}\left( p,q\right)
\geq \underline{\rho }_{A_{l}}(p,q).$

\textbf{Case 4.} Suppose $\max \left\{ \rho _{A_{j}}(p,q),j\neq m,l\right\} =%
\underline{\rho }_{A_{m}}(p,q)=\underline{\rho }_{A_{l}}(p,q)$ and

$\max \left\{ \underline{\tau }_{A_{m}}(p,q),\tau _{A_{j}}(p,q):\rho
_{A_{j}}(p,q)=\underline{\rho }_{A_{l}}(p,q),\text{ }j\neq m,l\right\} =\tau
_{2}<$ \underline{$\tau $}$_{A_{l}}(p,q).$

Hence by substituting $\left( \ref{2}\right) -\left( \ref{4}\right) ,\left( %
\ref{9}\right) ,\left( \ref{10}\right) ,\left( \ref{11}\right) ,\left( \ref%
{12}\right) $above all in $\left( \ref{1}\right) $,

for all $\left\vert z\right\vert =r\in S_{4}$ $\backslash $ $\left( \left[
0,1\right] \cup S_{1}\cup S_{2}\right) ,$ we have%
\begin{eqnarray*}
&&\exp _{p}\left[ \left( \underline{\tau }_{A_{l}}(p,q)-\varepsilon \right)
\left\{ \log _{q-1}\left( r\right) \right\} ^{\underline{\rho }_{A_{l}}(p,q)}%
\right] \\
&\leq &O\left( \exp _{p}\left[ \left( \tau _{2}+\varepsilon \right) \left\{
\log _{q-1}\left( r\right) \right\} ^{\underline{\rho }_{A_{l}}(p,q)}\right]
.\exp _{p}\left[ \left\{ \log _{q-1}\left( r\right) \right\} ^{\rho
_{f}\left( p,q\right) +\varepsilon }\right] \right) \\
&&+O\left( \exp _{p}\left[ \left\{ \log _{q-1}\left( r\right) \right\} ^{%
\underline{\rho }_{A_{l}}(p,q)-\varepsilon }\right] .\exp _{p}\left[ \left\{
\log _{q-1}\left( r\right) \right\} ^{\rho _{f}\left( p,q\right)
+\varepsilon }\right] \right) \\
&&+\exp _{p}\left[ \left( \underline{\tau }_{A_{m}}(p,q)+\varepsilon \right)
\left\{ \log _{q-1}\left( r\right) \right\} ^{\underline{\rho }_{A_{l}}(p,q)}%
\right] .\exp _{p}\left[ \left\{ \log _{q-1}\left( r\right) \right\} ^{\rho
_{f}\left( p,q\right) +\varepsilon }\right] \\
&&+\exp _{p}\left[ \left( \tau _{2}+\varepsilon \right) \left\{ \log
_{q-1}\left( r\right) \right\} ^{\underline{\rho }_{A_{l}}(p,q)}\right]
.\exp _{p}\left[ \left\{ \log _{q-1}\left( r\right) \right\} ^{\rho
_{f}\left( p,q\right) +\varepsilon }\right] .
\end{eqnarray*}

Choose sufficiently small $\varepsilon $ so that $0<2\varepsilon <\underline{%
\tau }_{A_{l}}(p,q)-\tau _{2}$ and

for all $\left\vert z\right\vert =r\in S_{4}$ $\backslash $ $\left( \left[
0,1\right] \cup S_{1}\cup S_{2}\right) ,r\rightarrow \infty $ we have%
\begin{equation*}
\exp _{p}\left[ \left( \underline{\tau }_{A_{l}}(p,q)-\tau _{2}-2\varepsilon
\right) \left\{ \log _{q-1}\left( r\right) \right\} ^{\underline{\rho }%
_{A_{l}}(p,q)-\varepsilon }\right] \leq \exp _{p}\left[ \left\{ \log
_{q-1}\left( r\right) \right\} ^{\rho _{f}\left( p,q\right) +\varepsilon }%
\right] ,
\end{equation*}

which implies%
\begin{equation*}
\underline{\rho }_{A_{l}}(p,q)\leq \rho _{f}\left( p,q\right) +2\varepsilon .
\end{equation*}

Since $\varepsilon >0$ is arbitrary, therefore $\rho _{f}\left( p,q\right)
\geq \underline{\rho }_{A_{l}}(p,q).$

In Case 2, Case 3, Case 4, for $A_{k+1}=0$ we obtain the same result $\rho
_{f}\left( p,q\right) \geq \underline{\rho }_{A_{l}}(p,q),$ by using the
method describe in Case 1.
\end{proof}

\begin{theorem}
\label{T2} Let $A_{0}(z),A_{1}(z),....,A_{k}(z),$ $A_{k+1}(z)$ be
meromorphic functions of generalize $\left( p,q\right) $-order$.$ If there
exist integers $l,m$ $\left( 0\leq l,m\leq k+1\right) $ and

the following conditions hold together :

(1) $0<\max \left\{ \underline{\rho }_{A_{m}}(p,q),\rho _{A_{j}}(p,q),j\neq
m,l\right\} =\rho \left( p,q\right) \leq \underline{\rho }%
_{A_{l}}(p,q)<\infty ;$

(2) \underline{$\tau $}$_{A_{l}}(p,q)>$\underline{$\tau $}$_{A_{m}}(p,q),$
when $\underline{\rho }_{A_{l}}(p,q)=\underline{\rho }_{A_{m}}(p,q);$

(3) $\dsum\limits_{\substack{ \rho _{A_{j}}\left( p,q\right) =\underline{%
\rho }_{A_{l}}(p,q)>0,  \\ j\neq m,l}}\tau _{A_{j}}(p,q)<$\underline{$\tau $}%
$_{A_{l}}(p,q)<+\infty ,$ when $\underline{\rho }_{A_{l}}(p,q)=\max \left\{
\rho _{A_{j}}(p,q),j\neq m,l\right\} .$

(4) $\dsum\limits_{\substack{ \rho _{A_{j}}\left( p,q\right) =\underline{%
\rho }_{A_{l}}(p,q)>0,  \\ j\neq m,l}}\tau _{A_{j}}(p,q)+\underline{\tau }%
_{A_{m}}(p,q)<$\underline{$\tau $}$_{A_{l}}(p,q)<+\infty ,$ when $\underline{%
\rho }_{A_{l}}(p,q)=\underline{\rho }_{A_{m}}(p,q)=\max \left\{ \rho
_{A_{j}}(p,q),j\neq m,l\right\} .$

(5) $\lambda \left( \frac{1}{A_{l}}\right) <\underline{\rho }%
_{A_{l}}(p,q)<\infty .$

Then every meromorphic solution $f\left( \neq 0\right) $ of equation $\left( %
\ref{1nh}\right) $ satisfies $\rho _{f}(p,q)\geq \underline{\rho }%
_{A_{l}}(p,q),$where $p\geq q\geq 2.$
\end{theorem}

\begin{proof}
Suppose that order of $f$ \ is finite, otherwise results hold trivially. Now
by $\left( \ref{1nh}\right) $ we obtain%
\begin{eqnarray}
T(r,A_{l}(z)) &=&m(r,A_{l}(z))+N(r,A_{l}(z))  \notag \\
&\leq &\sum_{j=0,j\neq m,l}^{k+1}m(r,A_{j}(z))+m(r,A_{l}(z))+\sum_{j=1,j\neq
l}^{k}m\left( r,\frac{f(z+c_{j})}{f(z+c_{l})}\right)  \notag \\
&&+m\left( r,\frac{f(z)}{f(z+c_{l})}\right) +m\left( r,\frac{1}{f(z+c_{l})}%
\right) +N(r,A_{l}(z))+O(1)  \notag \\
&\leq &\sum_{j=0,j\neq
m,l}^{k+1}T(r,A_{j}(z))+T(r,A_{m}(z))+N(r,A_{l}(z))+T\left( r,\frac{1}{%
f(z+c_{l})}\right)  \notag \\
&&+\sum_{j=1,j\neq l}^{k}m\left( r,\frac{f(z+c_{j})}{f(z+c_{l})}\right)
+m\left( r,\frac{f(z)}{f(z+c_{l})}\right) +O(1)  \notag \\
&\leq &\sum_{j=0,j\neq
m,l}^{k+1}T(r,A_{j}(z))+T(r,A_{m}(z))+N(r,A_{l}(z))+2T\left( r+\left\vert
c_{l}\right\vert ,f(z)\right)  \notag \\
&&+\sum_{j=1,j\neq l}^{k}m\left( r,\frac{f(z+c_{j})}{f(z+c_{l})}\right)
+m\left( r,\frac{f(z)}{f(z+c_{l})}\right) +O(1)  \label{13}
\end{eqnarray}

By using Lemma $\ref{3}$ for any given $\varepsilon >0,$ we obtain

\begin{equation}
m\left( r,\frac{f(z+c_{j})}{f(z+c_{l})}\right) =m\left( r,\frac{f(z)}{%
f(z+c_{l})}\right) =O\left( \exp _{p-1}\left[ \left\{ \log _{q-1}\left(
r\right) \right\} ^{\rho (p,q)+\varepsilon }\right] \right) ,\text{ }j\neq l.
\label{14}
\end{equation}

Again for sufficiently large $r$ and for above $\varepsilon $ by definition
we have%
\begin{equation}
N(r,A_{l}(z))\leq \exp _{p-1}\left[ \left\{ \log _{q-1}\left( r\right)
\right\} ^{\lambda \left( \frac{1}{A_{l}}\right) +\varepsilon }\right] .
\label{15}
\end{equation}

and we know by Lemma $\ref{2}$%
\begin{equation*}
2T\left( r+\left\vert c_{l}\right\vert ,f(z)\right) \leq 2T\left(
2r,f(z)\right) .
\end{equation*}

We consider four cases here.

\textbf{Case 1.} Consider the case when $\rho \left( p,q\right) <\underline{%
\rho }_{A_{l}}(p,q).$

Now for sufficiently large $r$ and for above $\varepsilon $ by definition we
have

\begin{equation}
T(r,A_{j}(z))\leq \exp _{p-1}\left[ \left\{ \log _{q-1}\left( r\right)
\right\} ^{\rho _{A_{j}}(p,q)+\varepsilon }\right] \leq \exp _{p-1}\left[
\left\{ \log _{q-1}\left( r\right) \right\} ^{\rho (p,q)+\varepsilon }\right]
,j\neq m,l  \label{16}
\end{equation}

and%
\begin{equation}
T(r,f)\leq \exp _{p-1}\left[ \left\{ \log _{q-1}\left( r\right) \right\}
^{\rho _{f}(p,q)+\varepsilon }\right] .  \label{17}
\end{equation}

Again for sufficiently large $r$ and for small $\varepsilon $ by definition
we have%
\begin{equation}
T(r,A_{l}(z))\geq \exp _{p-1}\left[ \left\{ \log _{q-1}\left( r\right)
\right\} ^{\underline{\rho }_{A_{l}}(p,q)-\varepsilon }\right] .  \label{18}
\end{equation}

Again using Lemma $\ref{7}$ by the definition of $\underline{\rho }_{A_{m}},$%
for any $\varepsilon >0,$ there exists a subset $S_{5}\subset \left(
1,+\infty \right) $ of infinite logarithmic measure such that for all $%
\left\vert z\right\vert =r\in S_{5},$ we have%
\begin{equation}
T(r,A_{m}(z))\leq \exp _{p-1}\left[ \left\{ \log _{q-1}\left( r\right)
\right\} ^{\underline{\rho }_{A_{m}}(p,q)+\varepsilon }\right] .  \label{18A}
\end{equation}

Hence by substituting $\left( \ref{14}\right) -\left( \ref{18}\right) $ in $%
\left( \ref{13}\right) $, for all $\left\vert z\right\vert =r\in S_{5}$ we
obtain%
\begin{eqnarray*}
\exp _{p-1}\left[ \left\{ \log _{q-1}\left( r\right) \right\} ^{\underline{%
\rho }_{A_{l}}(p,q)-\varepsilon }\right] &\leq &k\exp _{p-1}\left[ \left\{
\log _{q-1}\left( r\right) \right\} ^{\rho (p,q)+\varepsilon }\right] +\exp
_{p-1}\left[ \left\{ \log _{q-1}\left( r\right) \right\} ^{\underline{\rho }%
_{A_{m}}(p,q)+\varepsilon }\right] \\
&&+\exp _{p-1}\left[ \left\{ \log _{q-1}\left( r\right) \right\} ^{\lambda
\left( \frac{1}{A_{l}}\right) +\varepsilon }\right] +2\exp _{p-1}\left[
\left\{ \log _{q-1}\left( 2r\right) \right\} ^{\rho _{f}\left( p,q\right)
+\varepsilon }\right] \\
&&+O\left( \exp _{p-1}\left[ \left\{ \log _{q-1}\left( r\right) \right\}
^{\rho \left( p,q\right) +\varepsilon }\right] \right) .
\end{eqnarray*}

Choose sufficiently small $\varepsilon $ so that $0<3\varepsilon <\min
\left\{ \underline{\rho }_{A_{l}}(p,q)-\rho (p,q),\text{ }\underline{\rho }%
_{A_{l}}(p,q)-\lambda \left( \frac{1}{A_{l}}\right) \right\} $and for all $%
\left\vert z\right\vert =r\in S_{5},r\rightarrow \infty $ we have%
\begin{equation*}
\exp _{p-1}\left[ \left\{ \log _{q-1}\left( r\right) \right\} ^{\underline{%
\rho }_{A_{l}}(p,q)-2\varepsilon }\right] \leq \exp _{p-1}\left[ \left\{
\log _{q-1}\left( r\right) \right\} ^{\rho _{f}(p,q)+\varepsilon }\right] ,
\end{equation*}

which implies%
\begin{equation*}
\underline{\rho }_{A_{l}}\left( p,q\right) \leq \rho _{f}\left( p,q\right)
+3\varepsilon .
\end{equation*}

Since $\varepsilon >0$ is arbitrary, therefore $\rho _{f}\left( p,q\right)
\geq \underline{\rho }_{A_{l}}(p,q).$

If $A_{k+1}=0,$ then by substituting $\left( \ref{14}\right) -\left( \ref{16}%
\right) ,\left( \ref{18}\right) ,\left( \ref{18A}\right) $ in $\left( \ref%
{13}\right) $, for all $\left\vert z\right\vert =r\in S_{5}$ we obtain%
\begin{eqnarray*}
\exp _{p-1}\left[ \left\{ \log _{q-1}\left( r\right) \right\} ^{\underline{%
\rho }_{A_{l}}(p,q)-\varepsilon }\right] &\leq &(k-1)\exp _{p-1}\left[
\left\{ \log _{q-1}\left( r\right) \right\} ^{\rho (p,q)+\varepsilon }\right]
+\exp _{p-1}\left[ \left\{ \log _{q-1}\left( r\right) \right\} ^{\underline{%
\rho }_{A_{m}}(p,q)+\varepsilon }\right] \\
&&+\exp _{p-1}\left[ \left\{ \log _{q-1}\left( r\right) \right\} ^{\lambda
\left( \frac{1}{A_{l}}\right) +\varepsilon }\right] +O\left( \exp _{p-1}%
\left[ \left\{ \log _{q-1}\left( r\right) \right\} ^{\rho \left( p,q\right)
+\varepsilon }\right] \right) .
\end{eqnarray*}

Choose sufficiently small $\varepsilon $ so that $0<3\varepsilon <\min
\left\{ \underline{\rho }_{A_{l}}(p,q)-\rho (p,q),\text{ }\underline{\rho }%
_{A_{l}}(p,q)-\lambda \left( \frac{1}{A_{l}}\right) \right\} $and for all $%
\left\vert z\right\vert =r\in S_{5},r\rightarrow \infty $ we have%
\begin{equation*}
\exp _{p-1}\left[ \left\{ \log _{q-1}\left( r\right) \right\} ^{\underline{%
\rho }_{A_{l}}(p,q)-2\varepsilon }\right] \leq \exp _{p-1}\left[ \left\{
\log _{q-1}\left( r\right) \right\} ^{\rho _{f}(p,q)+\varepsilon }\right] ,
\end{equation*}

which implies%
\begin{equation*}
\underline{\rho }_{A_{l}}\left( p,q\right) \leq \rho _{f}\left( p,q\right)
+3\varepsilon .
\end{equation*}

Since $\varepsilon >0$ is arbitrary, therefore $\rho _{f}\left( p,q\right)
\geq \underline{\rho }_{A_{l}}(p,q).$

\textbf{Case 2.} \ Consider $\max \left\{ \rho _{A_{j}}(p,q),j\neq
m,l\right\} =\alpha <\underline{\rho }_{A_{m}}(p,q)=\underline{\rho }%
_{A_{l}}(p,q),$ \underline{$\tau $}$_{A_{l}}(p,q)>$\underline{$\tau $}$%
_{A_{m}}(p,q).$

Now for sufficiently large $r$ and for above $\varepsilon $ by definition we
have%
\begin{eqnarray}
T(r,A_{j}(z)) &\leq &\exp _{p-1}\left[ \left\{ \log _{q-1}\left( r\right)
\right\} ^{\rho _{A_{j}}(p,q)+\varepsilon }\right]  \notag \\
&\leq &\exp _{p-1}\left[ \left\{ \log _{q-1}\left( r\right) \right\}
^{\alpha +\varepsilon }\right] ,j\neq m,l  \label{19}
\end{eqnarray}

and 
\begin{equation}
T(r,A_{l}(z))\geq \exp _{p-1}\left[ \left( \underline{\tau }%
_{A_{l}}(p,q)-\varepsilon \right) \left\{ \log _{q-1}\left( r\right)
\right\} ^{\underline{\rho }_{A_{l}}(p,q)}\right] .  \label{20}
\end{equation}

Again using Lemma $\ref{8}$ by the definition of $\underline{\tau }_{A_{m}},$%
for any $\varepsilon >0,$ there exists a subset $S_{6}\subset \left(
1,+\infty \right) $ of infinite logarithmic measure such that for all $%
\left\vert z\right\vert =r\in S_{6},$ we have%
\begin{eqnarray}
T(r,A_{m}(z)) &\leq &\exp _{p-1}\left[ \left( \underline{\tau }%
_{A_{m}}(p,q)+\varepsilon \right) \left\{ \log _{q-1}\left( r\right)
\right\} ^{\underline{\rho }_{A_{m}}(p,q)}\right]  \notag \\
&=&\exp _{p-1}\left[ \left( \underline{\tau }_{A_{m}}(p,q)+\varepsilon
\right) \left\{ \log _{q-1}\left( r\right) \right\} ^{\underline{\rho }%
_{A_{l}}(p,q)}\right] .  \label{21}
\end{eqnarray}

Hence by substituting $\left( \ref{14}\right) ,\left( \ref{15}\right)
,\left( \ref{17}\right) ,\left( \ref{19}\right) -\left( \ref{21}\right) $ in 
$\left( \ref{13}\right) $, for all $\left\vert z\right\vert =r\in S_{6}$ we
obtain%
\begin{eqnarray*}
&&\exp _{p-1}\left[ \left( \underline{\tau }_{A_{l}}\left( p,q\right)
-\varepsilon \right) \left\{ \log _{q-1}\left( r\right) \right\} ^{%
\underline{\rho }_{A_{l}}(p,q)}\right] \\
&\leq &k\exp _{p-1}\left[ \left\{ \log _{q-1}\left( r\right) \right\} ^{\rho
\left( p,q\right) +\varepsilon }\right] +\exp _{p-1}\left[ \left( \underline{%
\tau }_{A_{m}}+\varepsilon \right) \left\{ \log _{q-1}\left( r\right)
\right\} ^{\underline{\rho }_{A_{l}}(p,q)}\right] \\
&&+\exp _{p-1}\left[ \left\{ \log _{q-1}\left( r\right) \right\} ^{\lambda
\left( \frac{1}{A_{l}}\right) +\varepsilon }\right] +2\exp _{p-1}\left[
\left\{ \log _{q-1}\left( 2r\right) \right\} ^{\rho _{f}\left( p,q\right)
+\varepsilon }\right] \\
&&+O\left( \exp _{p-1}\left[ \left\{ \log _{q-1}\left( r\right) \right\}
^{\rho \left( p,q\right) +\varepsilon }\right] \right) .
\end{eqnarray*}

Choose sufficiently small $\varepsilon $ so that

$0<2\varepsilon <\min \left\{ \underline{\rho }_{A_{l}}\left( p,q\right)
-\alpha ,\text{ }\underline{\tau }_{A_{l}}\left( p,q\right) -\underline{\tau 
}_{A_{m}}\left( p,q\right) ,\underline{\rho }_{A_{l}}\left( p,q\right)
-\lambda \left( \frac{1}{A_{l}}\right) \right\} $ and for all $\left\vert
z\right\vert =r\in S_{6}$ we have%
\begin{equation*}
\exp _{p-1}\left[ (\underline{\tau }_{A_{l}}\left( p,q\right) -\underline{%
\tau }_{A_{m}}\left( p,q\right) -2\varepsilon )\left\{ \log _{q-1}\left(
r\right) \right\} ^{\underline{\rho }_{A_{l}}(p,q)-\varepsilon }\right] \leq
\exp _{p-1}\left[ \left\{ \log _{q-1}\left( r\right) \right\} ^{\rho
_{f}(p,q)+\varepsilon }\right] ,
\end{equation*}

which implies%
\begin{equation*}
\underline{\rho }_{A_{l}}\left( p,q\right) \leq \rho _{f}\left( p,q\right)
+2\varepsilon .
\end{equation*}

Since $\varepsilon >0$ is arbitrary, therefore $\rho _{f}\left( p,q\right)
\geq \underline{\rho }_{A_{l}}\left( p,q\right) .$

If $A_{k+1}=0,$ then by substituting $\left( \ref{14}\right) ,\left( \ref{15}%
\right) ,\left( \ref{17}\right) ,\left( \ref{19}\right) -\left( \ref{21}%
\right) $ in $\left( \ref{13}\right) $, for all $\left\vert z\right\vert
=r\in S_{6}$ we obtain%
\begin{eqnarray*}
&&\exp _{p-1}\left[ \left( \underline{\tau }_{A_{l}}\left( p,q\right)
-\varepsilon \right) \left\{ \log _{q-1}\left( r\right) \right\} ^{%
\underline{\rho }_{A_{l}}(p,q)}\right] \\
&\leq &(k-1)\exp _{p-1}\left[ \left\{ \log _{q-1}\left( r\right) \right\}
^{\rho \left( p,q\right) +\varepsilon }\right] +\exp _{p-1}\left[ \left( 
\underline{\tau }_{A_{m}}\left( p,q\right) +\varepsilon \right) \left\{ \log
_{q-1}\left( r\right) \right\} ^{\underline{\rho }_{A_{l}}(p,q)}\right] \\
&&+\exp _{p-1}\left[ \left\{ \log _{q-1}\left( r\right) \right\} ^{\lambda
\left( \frac{1}{A_{l}}\right) +\varepsilon }\right] +O\left( \exp _{p-1}%
\left[ \left\{ \log _{q-1}\left( r\right) \right\} ^{\rho \left( p,q\right)
+\varepsilon }\right] \right) .
\end{eqnarray*}

Choose sufficiently small $\varepsilon $ so that

$0<2\varepsilon <\min \left\{ \underline{\rho }_{A_{l}}\left( p,q\right)
-\alpha ,\text{ }\underline{\tau }_{A_{l}}\left( p,q\right) -\underline{\tau 
}_{A_{m}}\left( p,q\right) ,\underline{\rho }_{A_{l}}\left( p,q\right)
-\lambda \left( \frac{1}{A_{l}}\right) \right\} $ and for all $\left\vert
z\right\vert =r\in S_{6}$ we have%
\begin{equation*}
\exp _{p-1}\left[ (\underline{\tau }_{A_{l}}\left( p,q\right) -\underline{%
\tau }_{A_{m}}\left( p,q\right) -2\varepsilon )\left\{ \log _{q-1}\left(
r\right) \right\} ^{\underline{\rho }_{A_{l}}(p,q)-\varepsilon }\right] \leq
\exp _{p-1}\left[ \left\{ \log _{q-1}\left( r\right) \right\} ^{\rho
_{f}\left( p,q\right) +\varepsilon }\right] ,
\end{equation*}

which implies%
\begin{equation*}
\underline{\rho }_{A_{l}}\left( p,q\right) \leq \rho _{f}\left( p,q\right)
+2\varepsilon .
\end{equation*}

Since $\varepsilon >0$ is arbitrary, therefore $\rho _{f}\left( p,q\right)
\geq \underline{\rho }_{A_{l}}\left( p,q\right) .$

\textbf{Case 3.} Suppose $\underline{\rho }_{A_{m}}(p,q)<\max \left\{ \rho
_{A_{j}}(p,q),\text{ }j\neq m,l\right\} =\underline{\rho }_{A_{l}}(p,q)$ and 
$\dsum\limits_{\rho _{A_{j}}(p,q)=\underline{\rho }_{A_{l}}(p,q),\text{ }%
j\neq m,l}\tau _{A_{j}}(p,q)=\tau _{1}<$ \underline{$\tau $}$%
_{A_{l}}(p,q)<+\infty .$

Then we have $\rho _{A_{j}}(p,q)=\underline{\rho }_{A_{l}}(p,q)$ with

$\dsum\limits_{\text{ }j\in \left\{ 1,2,...,k+1\right\} \backslash \left\{
m,l\right\} }\tau _{A_{j}}(p,q)=\tau _{1}<$ \underline{$\tau $}$%
_{A_{l}}(p,q)<+\infty $

and $\rho _{A_{j}}(p,q)<\underline{\rho }_{A_{l}}(p,q)$ for $j\in \left\{
1,2,...,k+1\right\} \backslash \left\{ 1,2,...,k+1\right\} \cup \left\{
m,l\right\} .$

Now for sufficiently large $r$ and for above $\varepsilon $ by definition we
have%
\begin{equation}
T(r,A_{j}(z))\leq \exp _{p-1}\left[ \left( \tau _{A_{j}}+\varepsilon \right)
\left\{ \log _{q-1}\left( r\right) \right\} ^{\underline{\rho }%
_{A_{l}}\left( p,q\right) }\right] ,j\in \left\{ 1,2,...,k+1\right\}
\backslash \left\{ m,l\right\}  \label{22}
\end{equation}

and%
\begin{equation}
T(r,A_{j}(z))\leq \exp _{p-1}\left[ \left\{ \log _{q-1}\left( r\right)
\right\} ^{\underline{\rho }_{A_{l}}\left( p,q\right) -\varepsilon }\right] ,%
\text{ }j\in \left\{ 1,2,...,k+1\right\} \backslash \left\{
1,2,...,k+1\right\} \cup \left\{ m,l\right\} .  \label{23}
\end{equation}

Hence by substituting $\left( \ref{14}\right) ,\left( \ref{15}\right)
,\left( \ref{17}\right) ,\left( \ref{18A}\right) ,\left( \ref{20}\right)
,\left( \ref{22}\right) ,\left( \ref{23}\right) $ in $\left( \ref{13}\right)
,$ for all $\left\vert z\right\vert =r\in S_{6}$ we obtain%
\begin{eqnarray*}
&&\exp _{p-1}\left[ \left( \underline{\tau }_{A_{l}}\left( p,q\right)
-\varepsilon \right) \left\{ \log _{q-1}\left( r\right) \right\} ^{%
\underline{\rho }_{A_{l}}\left( p,q\right) }\right] \\
&\leq &\dsum\limits_{\text{ }j\in \left\{ 1,2,...,k+1\right\} \backslash
\left\{ m,l\right\} }\exp _{p-1}\left[ \left( \underline{\tau }%
_{A_{j}}\left( p,q\right) +\varepsilon \right) \left\{ \log _{q-1}\left(
r\right) \right\} ^{\underline{\rho }_{A_{l}}\left( p,q\right) }\right] \\
&&+\dsum\limits_{\substack{ \text{ }j\in \left\{ 1,2,...,k+1\right\}  \\ %
\backslash \left\{ 1,2,...,k+1\right\} \cup \left\{ m,l\right\} }}\exp _{p-1}%
\left[ \left\{ \log _{q-1}\left( r\right) \right\} ^{\underline{\rho }%
_{A_{l}}\left( p,q\right) -\varepsilon }\right] \\
&&+\exp _{p-1}\left[ \left\{ \log _{q-1}\left( r\right) \right\} ^{%
\underline{\rho }_{A_{m}}\left( p,q\right) +\varepsilon }\right] +\exp _{p-1}%
\left[ \left\{ \log _{q-1}\left( r\right) \right\} ^{\lambda \left( \frac{1}{%
A_{l}}\right) +\varepsilon }\right] \\
&&+2\exp _{p-1}\left[ \left\{ \log _{q-1}\left( 2r\right) \right\} ^{\rho
_{f}\left( p,q\right) +\varepsilon }\right] +O\left( \exp _{p-1}\left[
\left\{ \log _{q-1}\left( r\right) \right\} ^{\rho \left( p,q\right)
+\varepsilon }\right] \right) \\
&\leq &\left( \tau _{1}+k\varepsilon \right) \exp _{p-1}\left[ \left\{ \log
_{q-1}\left( r\right) \right\} ^{\underline{\rho }_{A_{l}}\left( p,q\right) }%
\right] +O\left( \exp _{p-1}\left[ \left\{ \log _{q-1}\left( r\right)
\right\} ^{\underline{\rho }_{A_{l}}\left( p,q\right) -\varepsilon }\right]
\right) \\
&&+\exp _{p-1}\left[ \left\{ \log _{q-1}\left( r\right) \right\} ^{%
\underline{\rho }_{A_{m}}\left( p,q\right) +\varepsilon }\right] +\exp _{p-1}%
\left[ \left\{ \log _{q-1}\left( r\right) \right\} ^{\lambda \left( \frac{1}{%
A_{l}}\right) +\varepsilon }\right] \\
&&+2\exp _{p-1}\left[ \left\{ \log _{q-1}\left( 2r\right) \right\} ^{\rho
_{f}\left( p,q\right) +\varepsilon }\right] +O\left( \exp _{p-1}\left[
\left\{ \log _{q-1}\left( r\right) \right\} ^{\rho \left( p,q\right)
+\varepsilon }\right] \right) .
\end{eqnarray*}

Choose sufficiently small $\varepsilon $ so that $0<\varepsilon <\min
\left\{ \frac{\underline{\rho }_{A_{l}}\left( p,q\right) -\underline{\rho }%
_{A_{m}}\left( p,q\right) }{2},\text{ }\frac{\underline{\tau }_{A_{l}}\left(
p,q\right) -\tau _{1}}{k+1},\frac{\underline{\rho }_{A_{l}}\left( p,q\right)
-\lambda \left( \frac{1}{A_{l}}\right) }{2}\right\} $ and for all $%
\left\vert z\right\vert =r\in S_{5}$ we have%
\begin{equation*}
\exp _{p-1}\left[ (\underline{\tau }_{A_{l}}\left( p,q\right) -\tau
_{1}-(k+1)\varepsilon )\left\{ \log _{q-1}\left( r\right) \right\} ^{%
\underline{\rho }_{A_{l}}\left( p,q\right) -\varepsilon }\right] \leq \exp
_{p-1}\left[ \left\{ \log _{q-1}\left( r\right) \right\} ^{\rho _{f}\left(
p,q\right) +\varepsilon }\right] ,
\end{equation*}

which implies%
\begin{equation*}
\underline{\rho }_{A_{l}}\left( p,q\right) \leq \rho _{f}\left( p,q\right)
+2\varepsilon .
\end{equation*}

Since $\varepsilon >0$ is arbitrary, therefore $\rho _{f}\left( p,q\right)
\geq \underline{\rho }_{A_{l}}\left( p,q\right) .$

If $A_{k+1}=0,$ then by substituting $\left( \ref{14}\right) ,\left( \ref{15}%
\right) ,\left( \ref{18A}\right) ,\left( \ref{20}\right) ,\left( \ref{22}%
\right) ,\left( \ref{23}\right) $ in $\left( \ref{13}\right) $, for all $%
\left\vert z\right\vert =r\in S_{5}$ we obtain%
\begin{eqnarray*}
&&\exp _{p-1}\left[ \left( \underline{\tau }_{A_{l}}\left( p,q\right)
-\varepsilon \right) \left\{ \log _{q-1}\left( r\right) \right\} ^{%
\underline{\rho }_{A_{l}}\left( p,q\right) }\right] \\
&\leq &\left( \tau _{1}+(k-1)\varepsilon \right) \exp _{p-1}\left[ \left\{
\log _{q-1}\left( r\right) \right\} ^{\underline{\rho }_{A_{l}}\left(
p,q\right) }\right] +O\left( \exp _{p-1}\left[ \left\{ \log _{q-1}\left(
r\right) \right\} ^{\underline{\rho }_{A_{l}}\left( p,q\right) -\varepsilon }%
\right] \right) \\
&&+\exp _{p-1}\left[ \left\{ \log _{q-1}\left( r\right) \right\} ^{%
\underline{\rho }_{A_{m}}\left( p,q\right) +\varepsilon }\right] +\exp _{p-1}%
\left[ \left\{ \log _{q-1}\left( r\right) \right\} ^{\lambda \left( \frac{1}{%
A_{l}}\right) +\varepsilon }\right] \\
&&+O\left( \exp _{p-1}\left[ \left\{ \log _{q-1}\left( r\right) \right\}
^{\rho \left( p,q\right) +\varepsilon }\right] \right) .
\end{eqnarray*}

Choose sufficiently small $\varepsilon $ so that $0<\varepsilon <\min
\left\{ \frac{\underline{\rho }_{A_{l}}\left( p,q\right) -\underline{\rho }%
_{A_{m}}\left( p,q\right) }{2},\text{ }\frac{\underline{\tau }_{A_{l}}\left(
p,q\right) -\tau _{1}}{k},\frac{\underline{\rho }_{A_{l}}\left( p,q\right)
-\lambda \left( \frac{1}{A_{l}}\right) }{2}\right\} $ and for all $%
\left\vert z\right\vert =r\in S_{5}$ ,$r\rightarrow \infty $ we have%
\begin{equation*}
\exp _{p-1}\left[ (\underline{\tau }_{A_{l}}\left( p,q\right) -\tau
_{1}-k\varepsilon )\left\{ \log _{q-1}\left( r\right) \right\} ^{\underline{%
\rho }_{A_{l}}\left( p,q\right) -\varepsilon }\right] \leq \exp _{p-1}\left[
\left\{ \log _{q-1}\left( r\right) \right\} ^{\rho _{f}\left( p,q\right)
+\varepsilon }\right] ,
\end{equation*}

which implies%
\begin{equation*}
\underline{\rho }_{A_{l}}\left( p,q\right) \leq \rho _{f}\left( p,q\right)
+2\varepsilon .
\end{equation*}

Since $\varepsilon >0$ is arbitrary, therefore $\rho _{f}\left( p,q\right)
\geq \underline{\rho }_{A_{l}}\left( p,q\right) .$

\textbf{Case 4.} Suppose $\max \left\{ \rho _{A_{j}}(p,q),j\neq m,l\right\} =%
\underline{\rho }_{A_{m}}(p,q)=\underline{\rho }_{A_{l}}(p,q)$ and $%
\dsum\limits_{\rho _{A_{j}}(p,q)=\underline{\rho }_{A_{l}}(p,q)>0,\text{ }%
j\neq m,l}\tau _{A_{j}}(p,q)+\underline{\tau }_{A_{m}}(p,q)<$ \underline{$%
\tau $}$_{A_{l}}(p,q)<+\infty .$

Hence by substituting $\left( \ref{14}\right) ,\left( \ref{15}\right)
,\left( \ref{17}\right) ,\left( \ref{20}\right) ,\left( \ref{21}\right)
,\left( \ref{22}\right) ,\left( \ref{23}\right) $ in $\left( \ref{13}\right) 
$, for all $\left\vert z\right\vert =r\in S_{6}$ we obtain%
\begin{eqnarray*}
&&\exp _{p-1}\left[ \left( \underline{\tau }_{A_{l}}\left( p,q\right)
-\varepsilon \right) \left\{ \log _{q-1}\left( r\right) \right\} ^{%
\underline{\rho }_{A_{l}}\left( p,q\right) }\right] \\
&\leq &\dsum\limits_{\text{ }j\in \left\{ 1,2,...,k+1\right\} \backslash
\left\{ m,l\right\} }\exp _{p-1}\left[ \left( \underline{\tau }%
_{A_{j}}\left( p,q\right) +\varepsilon \right) \left\{ \log _{q-1}\left(
r\right) \right\} ^{\underline{\rho }_{A_{l}}\left( p,q\right) }\right] \\
&&+\dsum\limits_{\text{ }_{\substack{ \text{ }j\in \left\{
1,2,...,k+1\right\}  \\ \backslash \left\{ 1,2,...,k+1\right\} \cup \left\{
m,l\right\} }}}\exp _{p-1}\left[ \left\{ \log _{q-1}\left( r\right) \right\}
^{\underline{\rho }_{A_{l}}\left( p,q\right) -\varepsilon }\right] \\
&&+\exp _{p-1}\left[ \left( \underline{\tau }_{A_{m}}\left( p,q\right)
+\varepsilon \right) \left\{ \log _{q-1}\left( r\right) \right\} ^{%
\underline{\rho }_{A_{l}}\left( p,q\right) +\varepsilon }\right] +\exp _{p-1}%
\left[ \left\{ \log _{q-1}\left( r\right) \right\} ^{\lambda \left( \frac{1}{%
A_{l}}\right) +\varepsilon }\right] \\
&&+2\exp _{p-1}\left[ \left\{ \log _{q-1}\left( 2r\right) \right\} ^{\rho
_{f}+\varepsilon }\right] +O\left( \exp _{p-1}\left[ \left\{ \log
_{q-1}\left( r\right) \right\} ^{\rho \left( p,q\right) +\varepsilon }\right]
\right) \\
&\leq &\left( \tau _{1}+\underline{\tau }_{A_{m}}\left( p,q\right)
+(k+1)\varepsilon \right) \exp _{p-1}\left[ \left\{ \log _{q-1}\left(
r\right) \right\} ^{\underline{\rho }_{A_{l}}\left( p,q\right) }\right]
+O\left( \exp _{p-1}\left[ \left\{ \log _{q-1}\left( r\right) \right\} ^{%
\underline{\rho }_{A_{l}}\left( p,q\right) -\varepsilon }\right] \right) \\
&&+\exp _{p-1}\left[ \left\{ \log _{q-1}\left( r\right) \right\} ^{\lambda
\left( \frac{1}{A_{l}}\right) +\varepsilon }\right] +2\exp _{p-1}\left[
\left\{ \log _{q-1}\left( 2r\right) \right\} ^{\rho _{f}\left( p,q\right)
+\varepsilon }\right] \\
&&+O\left( \exp _{p-1}\left[ \left\{ \log _{q-1}\left( r\right) \right\}
^{\rho \left( p,q\right) +\varepsilon }\right] \right) .
\end{eqnarray*}

Choose sufficiently small $\varepsilon $ so that $0<\varepsilon <\min
\left\{ \frac{\underline{\tau }_{A_{l}}\left( p,q\right) -\underline{\tau }%
_{A_{m}}\left( p,q\right) -\tau _{1}}{k+2},\text{ }\frac{\underline{\rho }%
_{A_{l}}\left( p,q\right) -\lambda \left( \frac{1}{A_{l}}\right) }{2}%
\right\} $ and for all $\left\vert z\right\vert =r\in S_{6}$ we have%
\begin{equation*}
\exp _{p-1}\left[ (\underline{\tau }_{A_{l}}\left( p,q\right) -\tau _{1}-%
\underline{\tau }_{A_{m}}\left( p,q\right) -(k+2)\varepsilon )\left\{ \log
_{q-1}\left( r\right) \right\} ^{\underline{\rho }_{A_{l}}\left( p,q\right)
-\varepsilon }\right] \leq \exp _{p-1}\left[ \left\{ \log _{q-1}\left(
r\right) \right\} ^{\rho _{f}+\varepsilon }\right] ,
\end{equation*}

which implies%
\begin{equation*}
\underline{\rho }_{A_{l}}\left( p,q\right) \leq \rho _{f}\left( p,q\right)
+2\varepsilon .
\end{equation*}

Since $\varepsilon >0$ is arbitrary, therefore $\rho _{f}\left( p,q\right)
\geq \underline{\rho }_{A_{l}}\left( p,q\right) .$

If $A_{k+1}=0,$ then by substituting $\left( \ref{14}\right) ,\left( \ref{15}%
\right) ,\left( \ref{20}\right) ,\left( \ref{21}\right) ,\left( \ref{22}%
\right) ,\left( \ref{23}\right) $ in $\left( \ref{13}\right) $, for all $%
\left\vert z\right\vert =r\in S_{6}$ we obtain%
\begin{eqnarray*}
&&\exp _{p-1}\left[ \left( \underline{\tau }_{A_{l}}\left( p,q\right)
-\varepsilon \right) \left\{ \log _{q-1}\left( r\right) \right\} ^{%
\underline{\rho }_{A_{l}}\left( p,q\right) }\right] \\
&\leq &\left( \tau _{1}+\underline{\tau }_{A_{m}}\left( p,q\right)
+k\varepsilon \right) \exp _{p-1}\left[ \left\{ \log _{q-1}\left( r\right)
\right\} ^{\underline{\rho }_{A_{l}}\left( p,q\right) }\right] +O\left( \exp
_{p-1}\left[ \left\{ \log _{q-1}\left( r\right) \right\} ^{\underline{\rho }%
_{A_{l}}\left( p,q\right) -\varepsilon }\right] \right) \\
&&+\exp _{p-1}\left[ \left\{ \log _{q-1}\left( r\right) \right\} ^{\lambda
\left( \frac{1}{A_{l}}\right) +\varepsilon }\right] +O\left( \exp _{p-1}%
\left[ \left\{ \log _{q-1}\left( r\right) \right\} ^{\rho \left( p,q\right)
+\varepsilon }\right] \right) .
\end{eqnarray*}

Choose sufficiently small $\varepsilon $ so that $0<\varepsilon <\min
\left\{ \frac{\underline{\tau }_{A_{l}}\left( p,q\right) -\underline{\tau }%
_{A_{m}}\left( p,q\right) -\tau _{1}}{k+1},\text{ }\frac{\underline{\rho }%
_{A_{l}}\left( p,q\right) -\lambda \left( \frac{1}{A_{l}}\right) }{2}%
\right\} $ and for all $\left\vert z\right\vert =r\in S_{6}$ ,$r\rightarrow
\infty $ we have%
\begin{equation*}
\exp _{p-1}\left[ (\underline{\tau }_{A_{l}}\left( p,q\right) -\tau _{1}-%
\underline{\tau }_{A_{m}}\left( p,q\right) -(k+1)\varepsilon )\left\{ \log
_{q-1}\left( r\right) \right\} ^{\underline{\rho }_{A_{l}}\left( p,q\right)
-\varepsilon }\right] \leq \exp _{p-1}\left[ \left\{ \log _{q-1}\left(
r\right) \right\} ^{\rho _{f}+\varepsilon }\right] ,
\end{equation*}

which implies%
\begin{equation*}
\underline{\rho }_{A_{l}}\left( p,q\right) \leq \rho _{f}\left( p,q\right)
+2\varepsilon .
\end{equation*}

Since $\varepsilon >0$ is arbitrary, therefore $\rho _{f}\left( p,q\right)
\geq \underline{\rho }_{A_{l}}\left( p,q\right) .$
\end{proof}

\end{document}